\documentclass[reqno]{amsart}%
\usepackage{amssymb}
\usepackage{amsmath}
\usepackage{amsfonts}
\usepackage{hyperref}%
\usepackage{graphicx}
\providecommand{\U}[1]{\protect\rule{.1in}{.1in}}
\newtheorem{theorem}{Theorem}
\theoremstyle{plain}

\newtheorem{corollary}{Corollary}

\newtheorem{definition}{Definition}

\newtheorem{lemma}{Lemma}

\numberwithin{equation}{section}
\begin{document}
\title[ ]{The New Mixed Hypoexponential-G Family}
\author{Therar Kadri}
\address[T. Kadri]{ Northwestern Polytechnic}
\email{TKadri@nwpolytech.ca}
\author{Amina Halat}
\address[A. Halat]{ Lebanese International University}
\email{aminahalat1997@gmail.com}

\thanks{This paper is in final form and no version of it will be submitted for
publication elsewhere.}
\date{Nov 9, 2022}

\begin{abstract}
Through viewing out the literature, many generated distributions took a new
special form of probability density function (PDF) in which it is written as a
linear combination of n other distributions. Therefore, we define in this
paper a new type of distributions called "The New Mixed Distribution" form in
which a distribution is written as a linear combination of n others and derive
its characteristics. Second we construct "The New Mixed T-G family" a family
of distributions following another new defined type "The New Mixed T-G
Distribution". Third, we generate "The New Mixed Hypoexponential G-Family" out
of 6 new mixed distributions with characterizing their PDFs, CDFs, hazard rate
and reliability functions, MGF, and nth moment, and studying their maximum
likelihood estimator and method of moments.

\end{abstract}
\maketitle

\section{Introduction}

The Hypoexponential distribution is a continuous distribution that interferes
in everyday events by playing an important role in several fields such as
queuing theory, telegraphic engineering, and stochastic processes. It is named
Hypoexponential distribution as its coefficient is less than one and obtained
by adding $n$ $\geq2$ independent Exponential random variables. As a result,
the PDF of Hypoexponential distribution takes a special form that is a
summation of n Exponential distributions. See \cite{9}.

In the same manner, statisticians reached the same special form of
distributions when generating the PDF of ratio of 2 Hypoexponential
distribution (2014), PDF of summation of Extension Exponential distribution by
Kadri et al. (2022), See \cite{4}.

Based on what proceeds, in this paper, three are defined: New Mixed
Distribution form, New Mixed T-G Family, and the New Mixed Hypoexponential G-Family.

This paper introduces the new type of distributions "New Mixed Distribution"
in which a distribution PDF is written as a linear combination of n other
known distributions. Then, we derive the general form of CDF, hazard rate and
reliability functions, moment generating function, and moment of order k, and
generalize expressions for maximum likelihood estimator and method of moments
of these distributions. Next, we construct the "New Mixed T-G Distribution" a
type of distributions that follow the New Mixed Distribution form, with some
specific characteristics. Out of previous, we construct a new family of
distributions named the New Mixed T-G family and set its first example "The
New Mixed Hypoexponential G-Family" which consists of the Mixed
Hypoexponential Weibull distribution, Mixed Hypoexponential Frechet
distribution, Mixed Hypoexponential Pareto distribution, Mixed Hypoexponential
Power distribution, Mixed Hypoexponential Gumbel distribution, and Mixed
Hypoexponential Extreme Value distribution. Finally, the PDF with the
previously mentioned properties of these distributions are found in a direct
defined manner according to the properties of the New Mixed Distribution form.

\section{Some Preliminaries}

\subsection{Hypoexponential Distribution}

Hypoexponential Distribution is a continuous distribution. If a random
variable $X\sim Hypo\exp(\lambda_{1},\lambda_{2},...,\lambda_{n})$ then
$X=\sum\limits_{i=1}^{n}X_{i}$ where $X_{i}\sim Exp(\lambda_{i})\ $for
$i=1,2,...,n.$

\begin{theorem}
\label{difexp}Let $X_{i}$ $\sim Exp(\alpha_{i})$, $i=1,2,...,n$ be independent
random variables with $\alpha_{i}\neq\alpha_{j}$. Then $S_{n}=\sum
\limits_{i=1}^{n}X_{i}\sim hypo\exp(\overrightarrow{\alpha})$, where
$\overrightarrow{\alpha}=(\alpha_{1},\alpha_{2},...,\alpha_{n})$ has PDF%
\[
f_{S_{n}}(t)=\sum\limits_{i=1}^{n}\frac{f_{X_{i}}(t)}{P_{i}},\ t>0
\]
where $P_{i}=\prod\limits_{j=1,j\neq i}^{n}(1-\frac{\alpha_{i}}{\alpha_{j}}).$
\end{theorem}

Next is the graph of PDF of Hypoexponential distribution for some values of k
and $\lambda_{i}$.%
\begin{figure}
[ptb]
\begin{center}
\includegraphics[width=5in]%
{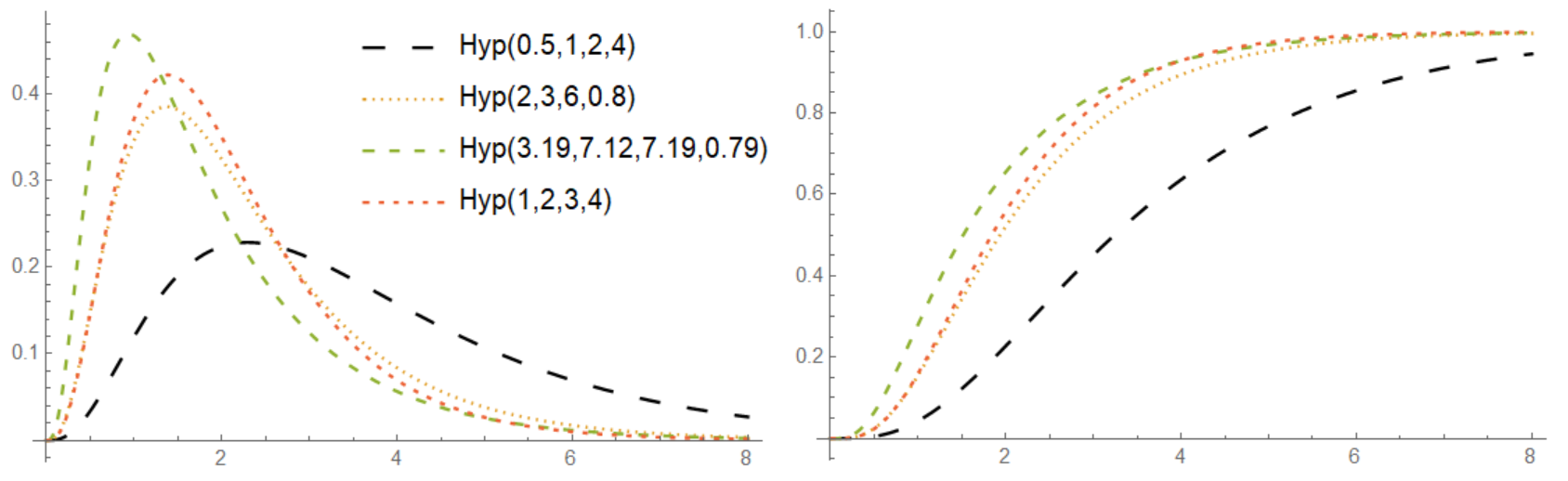}%
\caption{PDF and CDF of different distributions of hypoexp($\alpha_{1}%
,\alpha_{2},\alpha_{3},\alpha_{4})$}%
\end{center}
\end{figure}

\subsection{Definitions of some Random Variables}

Next, we present 6 different random variables with their corresponding PDF,
CDF, MGF, Reliability and Hazard Rate functions, and moment of order k.

\begin{definition}
The Weibull distribution is a continuous probability distribution. Let $X\sim
Weibull(\lambda,k,\gamma)$, then $X$ has the following PDF
\[
f(t)=\frac{\lambda}{k}(\frac{t-\gamma}{k})^{\lambda-1}e^{-(\frac{t-\gamma}%
{k})^{\lambda}}%
\]

where $f(t)>0$, $\lambda>0\ $is shape parameter, $k>0\ $is$\ $%
scale\ parameter, $-\infty<\gamma<\infty\ $is location parameter. The
2-parameter Weibull PDF is obtained by setting $\gamma=0$ and is given by%
\[
f(t)=\frac{\lambda}{k}(\frac{t}{k})^{\lambda-1}e^{-(\frac{t}{k})^{\lambda}}%
\]

with the following characteristics%
\[%
\begin{tabular}
[c]{ll}%
$F(t)=(1-e^{-(\frac{t}{k})^{\lambda}}),$ & $R(t)=e^{-(\frac{t}{k})^{\lambda}}%
$\\
$h(t)=\frac{\lambda}{k}\left(  \frac{t}{k}\right)  ^{\lambda-1},$ &
$\Phi(t)=\sum\limits_{n=0}^{\infty}\frac{t^{n}k^{n}}{n!}\Gamma(\frac
{1+n}{\lambda})$\\
$E[X^{n}]=k^{n}\Gamma(1+\frac{k}{\lambda}).$ &
\end{tabular}
\]
where $\Gamma\left(  s\right)  $ is Gamma distribution defined as
$\Gamma\left(  s\right)  =\int\limits_{0}^{\infty}t^{s-1}e^{-t}dt.\ $See
\cite{7}.
\end{definition}

\begin{definition}
The Frechet distribution is also known as Inverse Weibull distribution. Let
$X\sim Frechet(k,\lambda,\gamma)$, then $X$ has the following PDF
\[
f(t)=\frac{k}{\lambda}\left(  \frac{t-\gamma}{\lambda}\right)  ^{-1-k}%
e^{-\left(  \frac{t-\gamma}{\lambda}\right)  ^{-k}}%
\]

where $\lambda>0\ $is scale parameter, $k>0\ $is$\ $shape\ parameter,
$-\infty<\gamma<\infty\ $is location parameter. Taking the location parameter
$\gamma=0$ then the PDF of 2-parameter Frechet distribution is given by
\[
f(t)=\frac{k}{\lambda}\left(  \frac{t}{\lambda}\right)  ^{-1-k}e^{-\left(
\frac{t}{\lambda}\right)  ^{-k}}%
\]

and the CDF, Reliability Function, MGF, and nth moment are as the following%
\[%
\begin{tabular}
[c]{ll}%
$F(t)=e^{-\left(  \frac{t}{\lambda}\right)  ^{-k}},$ & $R(t)=e^{-k\left(
e^{\left(  \frac{k}{t}\right)  ^{\lambda}-1}\right)  ^{-\lambda}}$\\
$\Phi(t)=\sum\limits_{m=0}^{\infty}\frac{k^{m}t^{m}}{m!}\Gamma(\lambda
_{i}-\frac{m}{\lambda_{i}}),$ & $E[X^{n}]=\Gamma(1-\frac{n}{k}).$%
\end{tabular}
\]
for $\Gamma\left(  s\right)  =\int\limits_{0}^{\infty}t^{s-1}e^{-t}dt.$See
\cite{5}.
\end{definition}

\begin{definition}
Consider the random variable $X$ following Pareto distribution, then $X\sim
pareto(k,\lambda)$ of first kind has the following PDF
\[
f(t)=\frac{k\lambda^{k}}{t^{k+1}}%
\]

where $k>0$ is shape parameter and $\lambda>0$ is scale parameter. its
properties are given as%
\[%
\begin{tabular}
[c]{ll}%
$F(t)=(1-(\frac{\lambda}{t})^{k}),$ & $R(t)=\left(  \frac{\lambda}{t}\right)
^{k}$\\
$\Phi(t)=(k(-\lambda t)^{k}\Gamma(-k,-\lambda t),$ & $E[X^{n}]=\lambda
^{n}\frac{k}{k-n}.$%
\end{tabular}
\]
where $\Gamma\left(  n,\theta t\right)  $ is the Upper incomplete Gamma
function defined as $\Gamma\left(  n,\theta t\right)  =(n-1)!\sum
\limits_{k=0}^{n-1}\frac{(\theta t)^{k}}{k!}e^{-\theta t}.\ $See \cite{8}.
\end{definition}

\begin{definition}
Let $X$ follows Power distribution i.e. $X\sim power(k,\lambda)$ where k is
the domain parameter and $\lambda$ is the shape parameter. Its PDF is given
as
\[
f(t)=\lambda\ k^{\lambda}t^{-1+\lambda}%
\]

whereas its CDF, reliability function, MGF, nth moment are given as%
\[%
\begin{tabular}
[c]{ll}%
$F(t)=\left\{
\begin{tabular}
[c]{ll}%
$(\frac{t}{k})^{\lambda}$ & $0<t<\frac{1}{k}$\\
$1$ & $t>\frac{1}{k}$%
\end{tabular}
\right.  ,$ &
\begin{tabular}
[c]{ll}%
$R(t)=k^{\lambda}t^{\lambda}$ & $for\ 0<t\leq\frac{1}{k}$%
\end{tabular}
\\
$\Phi(t)=\frac{-\lambda\ k^{\lambda}}{\left(  -t\right)  ^{\lambda}}%
[\Gamma\left(  \lambda,\frac{-t}{k}\right)  -\Gamma\left(  \lambda\right)  ],
$ & $E[X^{n}]=k^{-n}\frac{\lambda}{n+\lambda}.$%
\end{tabular}
\]
where $\Gamma\left(  n,\theta t\right)  $ is the Upper incomplete Gamma
function defined as $\Gamma\left(  n,\theta t\right)  =(n-1)!\sum
\limits_{k=0}^{n-1}\frac{(\theta t)^{k}}{k!}e^{-\theta t}.$
\end{definition}

\begin{definition}
Consider the Gumbel distribution or Generalized Extreme Value distribution
Type-1, let $X\sim Gumbel(k,\lambda)$ then $X$ has the following PDF
\[
f(t)=\frac{1}{\lambda}e^{\frac{t-k}{\lambda}-e^{\frac{t-k}{\lambda}}}%
\]

where $\lambda>0$ is scale parameter and $k\in%
\mathbb{R}
$ is location parameter with the following properties%
\[%
\begin{tabular}
[c]{ll}%
$F(t)=e^{-e^{\frac{(t-k)}{\lambda}}},$ & $R(t)=(-e^{-e^{\frac{(t-k)}{\lambda}%
}})$\\
$\Phi(t)=\Gamma(1-\lambda t)e^{kt},$ & $E[X^{n}]=\int\limits_{R}\frac
{1}{\lambda}t^{n}e^{(\frac{t-k}{\lambda}-e^{\frac{t-k}{\lambda}})}dt.$%
\end{tabular}
\]
for $\Gamma\left(  s\right)  =\int\limits_{0}^{\infty}t^{s-1}e^{-t}dt.$ See
\cite{2}.
\end{definition}

\begin{definition}
Let X follows the Extreme Value distribution i.e. $X\sim$Extreme Value
Distribution $(k,\lambda)$ then it has the following PDF%
\[
f(t)=\frac{e^{\left(  -e^{\frac{-t+k}{\lambda}}+\frac{-t+k}{\lambda}\right)
}}{\lambda}%
\]

where $\lambda>0$ is scale parameter and $k\in%
\mathbb{R}
$ is location parameter. Its CDF, Reliability function, MGF, and nth moment
are as follows%
\[%
\begin{tabular}
[c]{ll}%
$F(t)=e^{-e^{-\frac{(t-k)}{\lambda}}},$ & $R(t)=\left(  1-e^{-e^{-\frac
{(t-k)}{\lambda}}}\right)  $\\
$\Phi(t)=\int e^{xt}\frac{e^{\left(  -e^{\frac{-x+k}{\lambda}}+\frac
{-x+k}{\lambda}\right)  }}{\lambda}dx,$ & $E[X^{n}]=\int\limits_{-\infty
}^{\infty}t^{n}\frac{e^{\left(  -e^{\frac{-t+k}{\lambda}}+\frac{-t+k}{\lambda
}\right)  }}{\lambda}dt.$%
\end{tabular}
\]

\end{definition}

\section{The New Mixed Distribution}

In this section we propose a new type of distributions called The New Mixed
Distribution in which a distribution is written as a linear combination of n
others. In addition, we derive the different properties for this type such as
PDF, CDF, MGF, nth moment, reliability and hazard rate functions, and also we
discuss some parameter estimations.

\begin{definition}
\label{mix}Let $X_{i},\ i\in I\subset%
\mathbb{N}
$ be vector of random variables that follow a distribution. A distribution Y
is called a Mixed Distribution of $X_{i}$ if it has PDF of the form
\[
f_{Y}(t)=\sum\limits_{i\in I}A_{i}f_{X_{i}}(t)
\]
where $A_{i}$ $\in%
\mathbb{R}
\ $and$\ \sum\limits_{i\in I}A_{i}=1$. Whenever $A_{i}\geq0$ this distribution
is the known mixture distribution of the random variables $X_{i} $ where
$A_{i}$ is the weight probability distribution along $X_{i}$. However, our new
distribution represents a general expression having $A_{i}$ $\in%
\mathbb{R}
.$
\end{definition}

\subsection{Properties of New Mixed distribution}

Regarding this special form of PDF we are able to generalize the CDF, the
reliability and hazard rate functions, MGF and nth moment of each distribution
having the New Mixed Distribution form.

\begin{theorem}
\label{cdf}Let $X_{i},$ $i\in I\subset%
\mathbb{N}
$ follows a known distribution and given a new distribution Y such that
\[
f_{Y}(t)=\sum\limits_{i\in I}A_{i}f_{X_{i}}(t)
\]
then its CDF is given as%
\[
F_{Y}(t)=\sum\limits_{i\in I}A_{i}F_{X_{i}}(t).
\]

\end{theorem}

\begin{proof}
Let $f_{Y}(t)=\sum\limits_{i\in I}A_{i}f_{X_{i}}(t)$, where $Y$ is a
continuous distribution then its CDF is gives as
\[
F_{Y}(t)=\int\limits_{-\infty}^{t}f_{Y}(x)dx=\int\limits_{-\infty}^{t}%
\sum\limits_{i\in I}A_{i}f_{X_{i}}(x)dx=\sum\limits_{i\in I}A_{i}%
\int\limits_{-\infty}^{t}f_{X_{i}}(x)dx=\sum\limits_{i\in I}A_{i}F_{X_{i}}(t)
\]
Now, suppose that Y is a discrete distribution then%
\[
F_{Y}(t)=\sum\limits_{x\leq t}f_{Y}(x)=\sum\limits_{x\leq t}\sum\limits_{i\in
I}A_{i}f_{X_{i}}(x)=\sum\limits_{i\in I}A_{i}\sum\limits_{x\leq t}f_{X_{i}%
}(x)=\sum\limits_{i\in I}A_{i}F_{X_{i}}(t).
\]

\end{proof}

\begin{theorem}
\label{reliability}Let $X_{i},$ $i\in I\subset%
\mathbb{N}
$ follows a known distribution and given a new distribution Y such that
\[
f_{Y}(t)=\sum\limits_{i\in I}A_{i}f_{X_{i}}(t)
\]
then its reliability function is given as%
\[
R_{Y}(t)=\sum\limits_{i\in I}A_{i}R_{X_{i}}(t).
\]

\end{theorem}

\begin{proof}
Let $f_{Y}(t)=\sum\limits_{i\in I}A_{i}f_{X_{i}}(t)$, first consider $Y$ is a
continuous distribution then the reliability function of Y is
\[
R_{Y}(t)=\int\limits_{t}^{\infty}f_{Y}(s)ds=\int\limits_{t}^{\infty}%
\sum\limits_{i\in I}A_{i}f_{X_{i}}(s)ds=\sum\limits_{i\in I}A_{i}%
\int\limits_{t}^{\infty}f_{X_{i}}(s)ds=\sum\limits_{i\in I}A_{i}R_{X_{i}}(t)
\]

Next, consider Y as a discrete distribution then
\[
R_{Y}(t)=\sum\limits_{x>t}f_{Y}(x)=\sum\limits_{x>t}\sum\limits_{i\in I}%
A_{i}f_{X_{i}}(x)=\sum\limits_{i\in I}A_{i}\sum\limits_{x>t}f_{X_{i}}%
(x)=\sum\limits_{i\in I}A_{i}R_{X_{i}}(t)
\]

\end{proof}

\begin{corollary}
\label{hazard}Let $X_{i},$ $i\in I\subset%
\mathbb{N}
$ follows a known distribution and given a new distribution Y such that
\[
f_{Y}(t)=\sum\limits_{i\in I}A_{i}f_{X_{i}}(t)
\]
then its hazard rate function is given as%
\[
h_{Y}(t)=\frac{\sum\limits_{i\in I}A_{i}F_{X_{i}}(t)}{\sum\limits_{i\in
I}A_{i}R_{X_{i}}(t)}.
\]

\end{corollary}

\begin{proof}
Let $f_{Y}(t)=\sum\limits_{i\in I}A_{i}f_{X_{i}}(t)$, then the hazard rate
function of Y is
\[
h_{Y}(t)=\frac{f_{Y}(t)}{R_{Y}(t)}=\frac{\sum\limits_{i\in I}A_{i}F_{X_{i}%
}(t)}{\sum\limits_{i\in I}A_{i}R_{X_{i}}(t)}%
\]

\end{proof}

\begin{theorem}
\label{mgf}Let $X_{i},$ $i\in I\subset%
\mathbb{N}
$ follows a known distribution and given a new distribution Y such that
\[
f_{Y}(t)=\sum\limits_{i\in I}A_{i}f_{X_{i}}(t)
\]
then its moment generating function MGF is given as%
\[
\Phi_{Y}(t)=\sum\limits_{i\in I}A_{i}\Phi_{X_{i}}(t).
\]

\end{theorem}

\begin{proof}
Let $f_{Y}(t)=\sum\limits_{i\in I}A_{i}f_{X_{i}}(t)$, first consider $Y$ is a
continuous distribution then its MGF is%
\[
\Phi_{Y}(t)=\int\limits_{-\infty}^{+\infty}e^{tx}f_{Y}(x)dx=\int
\limits_{-\infty}^{+\infty}e^{tx}\sum\limits_{i\in I}A_{i}f_{X_{i}}%
(x)dx=\sum\limits_{i\in I}A_{i}\int\limits_{-\infty}^{+\infty}e^{tx}f_{X_{i}%
}(x)dx=\sum\limits_{i\in I}A_{i}\Phi_{X_{i}}(t)
\]

while in case Y is discrete its MGF is
\[
\Phi_{Y}(t)=\sum\limits_{j\in J}e^{tx_{j}}f_{Y}(x_{j})=\sum\limits_{j\in
J}e^{tx_{j}}\sum\limits_{i\in I}A_{i}f_{X_{i}}(x_{j})=\sum\limits_{i\in
I}A_{i}\sum\limits_{j\in J}e^{tx_{j}}f_{X_{i}}(x_{j})=\sum\limits_{i\in
I}A_{i}\Phi_{X_{i}}(t)
\]

\end{proof}

\begin{theorem}
\label{nmoment}Let $X_{i},$ $i\in I\subset%
\mathbb{N}
$ follows a known distribution and given a new distribution Y such that
\[
f_{Y}(t)=\sum\limits_{i\in I}A_{i}f_{X_{i}}(t)
\]
then its moment of order k is
\[
E[Y^{k}]=\sum\limits_{i\in I}A_{i}E[X_{i}^{k}].
\]

\end{theorem}

\begin{proof}
Let $f_{Y}(t)=\sum\limits_{i\in I}A_{i}f_{X_{i}}(t)$, where$Y$ is a continuous
distribution then its moment of order k is%
\[
E[Y^{k}]=\int\limits_{-\infty}^{+\infty}x^{k}f_{Y}(x)dx=\int\limits_{-\infty
}^{+\infty}x^{k}\sum\limits_{i\in I}A_{i}f_{X_{i}}(x)dx=\sum\limits_{i\in
I}A_{i}\int\limits_{-\infty}^{+\infty}x^{k}f_{X_{i}}(x)dx=\sum\limits_{i\in
I}A_{i}E[X_{i}^{k}]
\]

similarly, in case Y is discrete then
\[
E[Y^{k}]=\sum\limits_{j\in J}x_{j}^{k}f_{Y}(x_{j})=\sum\limits_{j\in J}%
x_{j}^{k}\sum\limits_{i\in I}A_{i}f_{X_{i}}(x_{j})=\sum\limits_{i\in I}%
A_{i}\sum\limits_{j\in J}x_{j}^{k}f_{X_{i}}(x_{j})=\sum\limits_{i\in I}%
A_{i}E[X_{i}^{k}].
\]

\end{proof}

\begin{lemma}
Consider the mixed distribution $Y$ such that $f_{Y}(t)=\sum\limits_{i\in
I}A_{i}f_{X_{i}}(t)$ where $i\in I\subset%
\mathbb{N}
$ and $X_{i}$ are independent random variables. Then $\sum\limits_{i\in
I}A_{i}=1$.
\end{lemma}

\begin{proof}
Consider $f_{Y}(t)=\sum\limits_{i\in I}A_{i}f_{X_{i}}(t)$ the probability
density function of $Y$ then from Theorem \ref{cdf} $F_{Y}(t)=\sum
\limits_{i\in I}A_{i}F_{X_{i}}(t)$ is the CDF of Y. Then,%
\begin{align*}
\underset{t\rightarrow\infty}{\lim}F_{Y}(t)  & =1\\
\underset{t\rightarrow\infty}{\lim}\sum\limits_{i\in I}A_{i}F_{X_{i}}(t)  &
=1\\
\sum\limits_{i\in I}A_{i}\underset{t\rightarrow\infty}{\lim}F_{X_{i}}(t)  & =1
\end{align*}
as $F_{X_{i}}(t)$ is the CDF of $X_{i}\ $then $\underset{t\rightarrow\infty
}{\lim}F_{X_{i}}(t)=1.$ Therefore,
\[
\sum\limits_{i\in I}A_{i}=1.
\]

\end{proof}

Therefore, in this part we defined the New Mixed Distribution type and
generalized its CDF, reliability function, hazard rate function, MGF and nth moment.

\subsection{Parameter Estimation of New Mixed distribution}

In addition to what proceeds, we present the following methods to generalize
the estimation of parameters of any distribution of this form.

\subsubsection{Maximum Likelihood Estimation}

\begin{theorem}
\label{ml}Given a mixed distribution $Y$ of parameter $\overrightarrow{\theta
}$ such that $f_{Y}(x)=\sum\limits_{j\in J}A_{j}f_{X_{j}}(x)$ where $X_{j},$
$j\in J\subset%
\mathbb{N}
$ are independent random variables that follow a known distribution. Given
independent observations $x_{1},x_{2},...,x_{n}$ then the maximum likelihood
estimator $\overset{\wedge}{\theta}$ is that which maximizes the likelihood
function
\[
L(x_{1},x_{2},...,x_{n};\overrightarrow{\theta})=f_{Y}(x,\overrightarrow
{\theta})=f_{Y}(x_{1},\overrightarrow{\theta})f_{Y}(x_{2},\overrightarrow
{\theta})...f_{Y}(x_{n},\overrightarrow{\theta})
\]
is obtained by solving the equality
\[
\sum\limits_{i=1}^{n}\frac{\sum\limits_{j\in J}\frac{\partial}{\partial\theta
}A_{j}f_{X_{j}}(x_{i}|\overrightarrow{\theta})}{\sum\limits_{j\in J}%
A_{j}f_{X_{j}}(x_{i}|\overrightarrow{\theta})}=0.
\]
Given a distribution Y of parameter $\overrightarrow{\theta}$ such that
$f_{Y}(x)=\sum\limits_{j\in J}A_{j}f_{X_{j}}(x)$ where $X_{j},\ j\in J\subset%
\mathbb{N}
$ be independent random variables that follow a known distribution. Suppose
that the random sample $x_{1},x_{2},...,x_{n}$ is taken from the distribution
then to find the maximum likelihood estimate of $\theta$ given that%
\[
L(x_{1},x_{2},...,x_{n};\overrightarrow{\theta})=\prod\limits_{i=1}^{n}%
f_{Y}(x_{i}|\overrightarrow{\theta})
\]

\end{theorem}

\begin{proof}
second,
\begin{align*}
\ln L(x_{1},x_{2},...,x_{n};\overrightarrow{\theta})  & =\ln\prod
\limits_{i=1}^{n}f_{Y}(x_{i}|\overrightarrow{\theta})\\
& =\sum\limits_{i=1}^{n}\ln f_{Y}(x_{i}|\overrightarrow{\theta})\\
& =\sum\limits_{i=1}^{n}\ln\sum\limits_{j\in J}A_{j}f_{X_{j}}(x_{i}%
|\overrightarrow{\theta})
\end{align*}

third for parameter $\theta$%
\begin{align*}
\frac{\partial\ln L(x_{1},x_{2},...,x_{n};\theta)}{\partial\theta}  &
=\frac{\partial}{\partial\theta}\sum\limits_{i=1}^{n}\ln\sum\limits_{j\in
J}A_{j}f_{X_{j}}(x_{i}|\theta)\\
& =\sum\limits_{i=1}^{n}\frac{\partial}{\partial\theta}\ln\sum\limits_{j\in
J}A_{j}f_{X_{j}}(x_{i}|\theta)\\
& =\sum\limits_{i=1}^{n}\frac{\frac{\partial}{\partial\theta}\sum\limits_{j\in
J}A_{j}f_{X_{j}}(x_{i}|\theta)}{\sum\limits_{j\in J}A_{j}f_{X_{j}}%
(x_{i}|\theta)}\\
& =\sum\limits_{i=1}^{n}\frac{\sum\limits_{j\in J}\frac{\partial}%
{\partial\theta}A_{j}f_{X_{j}}(x_{i}|\theta)}{\sum\limits_{j\in J}%
A_{j}f_{X_{j}}(x_{i}|\theta)}%
\end{align*}

fourth
\[
\sum\limits_{i=1}^{n}\frac{\sum\limits_{j\in J}\frac{\partial}{\partial\theta
}A_{j}f_{X_{j}}(x_{i}|\theta)}{\sum\limits_{j\in J}A_{j}f_{X_{j}}(x_{i}%
|\theta)}=0.
\]

\end{proof}

\subsubsection{Method of Moments}

\begin{theorem}
\label{moments}Suppose the problem is to estimate n unknown parameters
$\theta_{1},...,\theta_{n}$ characterizing the distribution
\[
f_{Y}(x|\mathbf{\theta})=\sum\limits_{i\in I}A_{i}f_{X_{i}}(x;\mathbf{\theta})
\]
of the random variable $Y(\theta_{1},...,\theta_{n})\ $where $X_{i},$ $i\in
I\subset%
\mathbb{N}
$ are independent random variables that follow a known distribution. Then, the
first n moments are expressed as follows%
\begin{align*}
\mu_{1}  & =g_{1}(\theta_{1},...,\theta_{n})=E[Y^{1}]=\sum\limits_{i\in
I}A_{i}E[X_{i}^{1}]\\
\mu_{2}  & =g_{2}(\theta_{1},...,\theta_{n})=E[Y^{2}]=\sum\limits_{i\in
I}A_{i}E[X_{i}^{2}]\\
& \vdots\\
\mu_{n}  & =g_{n}(\theta_{1},...,\theta_{n})=E[Y^{n}]=\sum\limits_{i\in
I}A_{i}E[X_{i}^{n}]
\end{align*}

suppose a sample of size $m$ is drawn, resulting in the values $y_{1}%
,y_{2},...,y_{m}$ for $k=1,2,...,n$ let
\[
\overset{\wedge}{\mu}_{k}=\frac{1}{m}\sum\limits_{j=1}^{m}y_{j}^{k}%
\]
be the k-th sample moment, an estimate of $\mu_{k}$. The method of moments
estimator for $\theta_{1},...,\theta_{n}$ denoted by $\overset{\wedge}{\theta
}_{1},...,\overset{\wedge}{\theta}_{n}$is defined as the solution to the
equation%
\[%
\begin{tabular}
[c]{ll}%
$\overset{\wedge}{\mu}_{k}=g_{k}(\overset{\wedge}{\theta}_{1},...,\overset
{\wedge}{\theta}_{n})$ & $k=1,2,...,n.$%
\end{tabular}
\]

\end{theorem}

Finally, in this section we were able to generate a new type of distributions
called " The New Mixed Distribution" and examine different properties
regarding its form such as PDF, MGF, reliability and hazard rate functions,
and moment of order k then generalize 2 methods of estimations which are
maximum likelihood estimation and method of moments.

\section{The New Mixed T-G Family}

In this section, we generate from the mixed distribution "T" a new type of
distributions named the New Mixed T-G Distribution which is a mixed
distribution, leading to construct The New Mixed T-G Family of same mother
mixed distribution T.

\subsection{The New Mixed T-G Distribution}

Next, we define the second New type of distributions which is "The New Mixed
T-G Distribution" and generalize some of its properties.

\begin{theorem}
Given a random variable $Z$ with a probability density function $f_{Z}%
(t)=\sum\limits_{i\in I}A_{i}f_{X_{i}}(t),$ where $X_{i},$ $i\in I\subset%
\mathbb{N}
$ be vector of random variables and $A_{i}$ $\in%
\mathbb{R}
$. Let $Y=g(X)$ be a 1-1 function of a random variable $X$. Then
\[
\sum\limits_{i\in I}A_{i}f_{g(X_{i})}(y)
\]
is a valid distribution generated from the mixed distribution of $X_{i}$.

\begin{proof}
The aim is to proof that $\sum\limits_{i\in I}A_{i}f_{g(X_{i})}(t)$ is a valid
PDF. First,suppose that $g(X_{i})$ is continuous. We start by proving the
expression is positive. We have
\[
f_{g(X_{i})}(y)=\sum_{j=1}^{k}f_{X_{i}}\left[  w_{j}(y)\right]  \left\vert
J_{j}\right\vert =\sum_{j=1}^{k}f_{X_{i}}\left[  x_{j}\right]  \left\vert
J_{j}\right\vert
\]
where $J_{i}=w_{i}^{\prime}\left(  y\right)  $ Now, $\sum\limits_{i\in I}%
A_{i}f_{g(X_{i})}(y)=\sum\limits_{i\in I}A_{i}\sum_{j=1}^{k}f_{X_{i}}\left[
x_{j}\right]  \left\vert J_{j}\right\vert =\sum_{j=1}^{k}\sum\limits_{i\in
I}A_{i}f_{X_{i}}\left[  x_{j}\right]  \left\vert J_{j}\right\vert =\sum
_{j=1}^{k}f_{Z}\left(  x_{j}\right)  \left\vert J_{j}\right\vert $. Now, since
$f_{Z}\left(  x\right)  $ is a PDF of $Z$, then $f_{Z}\left(  x\right)  \geq0$
and $\left\vert J_{j}\right\vert \geq0$, thus $\sum_{j=1}^{k}f_{Z}\left(
x_{j}\right)  \left\vert J_{i}\right\vert \geq0$. Next, we need to prove that
$\int_{%
\mathbb{R}
}f_{T}(y)dy=1$. We have $\int_{%
\mathbb{R}
}f_{T}(y)dy=\int_{%
\mathbb{R}
}\sum\limits_{i\in I}A_{i}f_{g(X_{i})}(y)dy=\sum\limits_{i\in I}A_{i}\int_{%
\mathbb{R}
}f_{g(X_{i})}(y)dy=\sum\limits_{i\in I}A_{i}$ as $\int_{%
\mathbb{R}
}f_{g(X_{i})}(y)dy=1$. Also from Definition \ref{mix} $\sum\limits_{i\in
I}A_{i}=1.$ Hence $\int_{%
\mathbb{R}
}f_{T}(y)dy=1$. Therefore, $\sum\limits_{i\in I}A_{i}f_{g(X_{i})}(t)$ is a
valid PDF.

Second, we will reprove the previous but for $g(X_{i})$ is discrete random
variable. First,
\[
f_{g(X_{i})}(y)=\sum_{j=1}^{k}f_{X_{i}}\left[  w_{j}(y)\right]  =\sum
_{j=1}^{k}f_{X_{i}}\left[  x_{j}\right]
\]
And, $\sum\limits_{i\in I}A_{i}f_{g(X_{i})}(y)=\sum\limits_{i\in I}A_{i}%
\sum_{j=1}^{k}f_{X_{i}}\left[  x_{j}\right]  =\sum_{j=1}^{k}\sum\limits_{i\in
I}A_{i}f_{X_{i}}\left[  x_{j}\right]  =\sum_{j=1}^{k}f_{Z}\left(
x_{j}\right)  $. As $f_{Z}\left(  x\right)  $ is a PDF of $Z$, then
$f_{Z}\left(  x\right)  \geq0$, thus $\sum_{j=1}^{k}f_{Z}\left(  x_{j}\right)
\geq0$.

Still to prove that $\sum\limits_{y}f_{T}(y)=1$. We have $\sum\limits_{y}%
f_{T}(y)=\sum\limits_{y}\sum\limits_{i\in I}A_{i}f_{g(X_{i})}(y)=\sum
\limits_{i\in I}A_{i}\sum\limits_{y}f_{g(X_{i})}(y)=\sum\limits_{i\in I}A_{i}$
as $\sum\limits_{y}f_{g(X_{i})}(y)=1$. Also from Definition \ref{mix}
$\sum\limits_{i\in I}A_{i}=1.$ Hence $\sum\limits_{y}f_{T}(y)=1$. Therefore,
$\sum\limits_{i\in I}A_{i}f_{g(X_{i})}(t)$ is a valid PDF.
\end{proof}
\end{theorem}

\begin{definition}
Given our Mixed distribution of $X_{i},$ having a PDF
\[
\sum\limits_{i\in I}A_{i}f_{X_{i}}(t)
\]
that follows a distribution named "$T$" and suppose that $g(X_{i})$ is a known
unique distribution $g(X_{i})$ named "G". We denote the Mixed T-G distribution
generated from T as the distribution having a PDF
\[
\sum\limits_{i\in I}A_{i}f_{g(X_{i})}(y).
\]

\end{definition}

\subsection{Properties of New Mixed T-G Distribution}

We point out that the mixed T-G distribution having a PDF $\sum\limits_{i\in
I}A_{i}f_{g(X_{i})}(y)$ is a mixed distribution thus we may consider that all
the properties in Theorems \ref{cdf}, \ref{reliability}, \ref{hazard},
\ref{mgf}, \ref{nmoment} may be used for our new distribution. Thus we state
some important results. The CDF of our new Mixed T-G distribution is
\[
\sum\limits_{i\in I}A_{i}F_{g\left(  X_{i}\right)  }(y),
\]
the reliability function is
\[
\sum\limits_{i\in I}A_{i}R_{g\left(  X_{i}\right)  }(y)
\]
the hazard rate function
\[
\frac{\sum\limits_{i\in I}A_{i}f_{g\left(  X_{i}\right)  }(t)}{\sum
\limits_{i\in I}A_{i}R_{g\left(  X_{i}\right)  }(t)}%
\]
MGF is
\[
\sum\limits_{i\in I}A_{i}\Phi_{g\left(  X_{i}\right)  }(t)
\]
and moment of order k is
\[
\sum\limits_{i\in I}A_{i}E[g\left(  X_{i}\right)  ^{k}].
\]

In this part, we generate the second type of distributions "The New Mixed T-G
distribution" out of a mixed distribution T of $X_{i}$ and a transformation G
of $X_{i}$ then it is a mixed distribution of transformation of $X_{i}.$
Moreover, the CDF, MGF, reliability and hazard rate functions, and moment of
order k for this type are generalized. Finally, We end this Section to point
out that the New Mixed T-G Family is generated by fixing a parent mixed
distribution T of $X_{i}$ then substituting $X_{i}$ by its transformations G
to obtain different Mixed T-G Distributions belonging to the same family.

\section{The New Mixed Hypoexponential-G Family}

In this section we adopt the Hypoexponential distribution with different
parameters to be the parent distribution T of the Mixed Hypoexponential T-G
Family. This can be defined as the Hypoexponential distribution is a Mixed
distribution of the exponential distribution from Definition \ref{mix}%
.Therefore, we generate some distributions and examine deep the properties
obtained for these Mixed distributions.

\subsection{The Mixed Hypoexponential Weibull Distribution}

In this section we introduce a new distribution denoted by the Mixed
Hypoexponential Weibull distribution. This distribution is generated from the
Hypoexponential distribution with different parameters with a Weibull
transform distribution.

We start from the Hypoexponential distribution with different parameters
$S_{n}\sim hypo\exp(\overrightarrow{\alpha})$, where $\overrightarrow{\alpha
}=\left(  \alpha_{1},\alpha_{2},...,\alpha_{n}\right)  $ and $\alpha_{i}%
\neq\alpha_{j}$. The PDF of $S_{n}$ is given from Theorem \ref{difexp} as
\[
f_{S_{n}}(t)=\sum\limits_{i=1}^{n}\frac{f_{X_{i}}(t)}{P_{i}},\ t>0
\]
where $X_{i}$ $\sim Exp(\alpha_{i})$ and $P_{i}=\prod\limits_{j=1,j\neq i}%
^{n}(1-\frac{\alpha_{i}}{\alpha_{j}})$.

On the other hand we use the transform distribution of $X_{i}\sim
Exp(\alpha_{i})$, as $Y_{i}=X_{i}^{c}\sim Weibull(\frac{1}{c},\frac{1}%
{\alpha_{i}^{c}})$ where $c>0.$ Now, suppose that $k=\frac{1}{c}$,
$\lambda_{i}=\frac{1}{\alpha_{i}^{c}}$, then $\alpha=\frac{1}{\lambda^{1/c}%
}=\frac{1}{\lambda^{k}}$ and $c=\frac{1}{k}$ and we may write $Y_{i}\sim
Weibull(k,\lambda_{i})$, $i=1,2,...,n$. Thus we write $P_{i}=\prod
\limits_{j=1,j\neq i}^{n}(1-\frac{\alpha_{i}}{\alpha_{j}})=\prod
\limits_{j=1,j\neq i}^{n}(1-\frac{\frac{1}{\lambda_{i}^{k}}}{\frac{1}%
{\lambda_{j}^{k}}})=\prod\limits_{j=1,j\neq i}^{n}(1-\frac{\lambda_{j}^{k}%
}{\lambda_{i}^{k}})=\prod\limits_{j=1,j\neq i}^{n}(1-\left(  \frac{\lambda
_{j}}{\lambda_{i}}\right)  ^{k})$, we call this $PW_{i}$ which is transformed
from $P_{i}$. Therefore, we obtain our new distribution, the Mixed
Hypoexponential Weibull distribution $Z\sim MHW(k,\lambda_{1},\lambda
_{2},...,\lambda_{n}),k,\lambda_{i}>0$ with PDF
\[
f_{Z}(t)=\sum\limits_{i=1}^{n}\frac{1}{PW_{i}}f_{Y_{i}}(t)
\]
where $Y_{i}$ $\sim Weibull(k,\lambda_{i})$ and $PW_{i}=\prod
\limits_{j=1,j\neq i}^{n}(1-\left(  \frac{\lambda_{j}}{\lambda_{i}}\right)
^{k}).$

\subsubsection{The Properties of Mixed Hypoexponential Weibull Distribution}

Here we generalize the CDF, MGF, Reliability and Hazard Rate functions, for
the Mixed Hypoexponential Weibull Distribution.

\begin{theorem}
\label{propw}Let $Z\sim MHW(k,\lambda_{1},\lambda_{2},...,\lambda_{n})$ then
according to Theorems \ref{cdf}, \ref{reliability}, \ref{hazard}, \ref{mgf},
and \ref{nmoment} $Z$ has the following properties%
\[%
\begin{tabular}
[c]{ll}%
$F_{X}(t)=\sum\limits_{i=1}^{n}\frac{1}{PW_{i}}F_{Y_{i}}(t)$ & $R_{X}%
(t)=\sum\limits_{i=1}^{n}\frac{1}{PW_{i}}R_{Y_{i}}(t)$\\
$h_{X}(t)=\frac{\sum\limits_{i=1}^{n}\frac{1}{PW_{i}}f_{Y_{i}}(t)}%
{\sum\limits_{i=1}^{n}\frac{1}{PW_{i}}R_{Y_{i}}(t)}$ & $\Phi_{Z}%
(t)=\sum\limits_{i=1}^{n}\frac{1}{PW_{i}}\Phi_{Y_{i}}(t)$\\
$E[Z^{h}]=\sum\limits_{i=1}^{n}\frac{1}{PW_{i}}E[Y_{i}^{k}]$ &
\end{tabular}
\]

where $Y_{i}$ $\sim Weibull(k,\lambda_{i})$ and $PW_{i}=\prod
\limits_{j=1,j\neq i}^{n}(1-\left(  \frac{\lambda_{j}}{\lambda_{i}}\right)
^{k}).$
\end{theorem}

\begin{corollary}
Let $Z\sim MHW(k,\lambda_{1},\lambda_{2},...,\lambda_{n})$ then according to
Theorem \ref{propw} the CDF, reliability and hazard rate functions, MGF and
moment of order h of $Z$ are%
\[%
\begin{tabular}
[c]{ll}%
$F_{Z}(t)=\sum\limits_{i=1}^{n}\frac{(1-e^{-(\frac{t}{k})^{\lambda_{i}}}%
)}{\prod\limits_{j=1,j\neq i}^{n}(1-\left(  \frac{\lambda_{j}}{\lambda_{i}%
}\right)  ^{k})}$ & $R_{Z}(t)=\sum\limits_{i=1}^{n}\frac{e^{-(\frac{t}%
{k})^{\lambda_{i}}}}{\prod\limits_{j=1,j\neq i}^{n}(1-\left(  \frac
{\lambda_{j}}{\lambda_{i}}\right)  ^{k})}$\\
$h_{Z}(t)=\frac{\sum\limits_{i=1}^{n}\frac{\frac{\lambda_{i}}{k}(\frac{t}%
{k})^{\lambda_{i}-1}e^{-(\frac{t}{k})^{\lambda_{i}}}}{\prod\limits_{j=1,j\neq
i}^{n}(1-\left(  \frac{\lambda_{j}}{\lambda_{i}}\right)  ^{k})}}%
{\sum\limits_{i=1}^{n}\frac{e^{-(\frac{t}{k})^{\lambda_{i}}}}{\prod
\limits_{j=1,j\neq i}^{n}(1-\left(  \frac{\lambda_{j}}{\lambda_{i}}\right)
^{k})}}$ & $\Phi_{Z}(t)=\sum\limits_{i=1}^{n}\frac{\sum\limits_{n=0}^{\infty
}\frac{t^{n}k^{n}}{n!}\Gamma(\frac{1+n}{\lambda_{i}})}{\prod\limits_{j=1,j\neq
i}^{n}(1-\left(  \frac{\lambda_{j}}{\lambda_{i}}\right)  ^{k})}$\\
$E[Z^{h}]=\sum\limits_{i=1}^{n}\frac{k^{h}\Gamma(1+\frac{k}{\lambda_{i}}%
)}{\prod\limits_{j=1,j\neq i}^{n}(1-\left(  \frac{\lambda_{j}}{\lambda_{i}%
}\right)  ^{k})}.$ &
\end{tabular}
\]

\end{corollary}

Next, we give some different graphs of PDF and CDF for some values of k and
$\lambda_{i}$, showing the flexibility of adding extra parameters to a new
distribution.%
\begin{figure}
[ptbh]
\begin{center}
\includegraphics[width=5in]%
{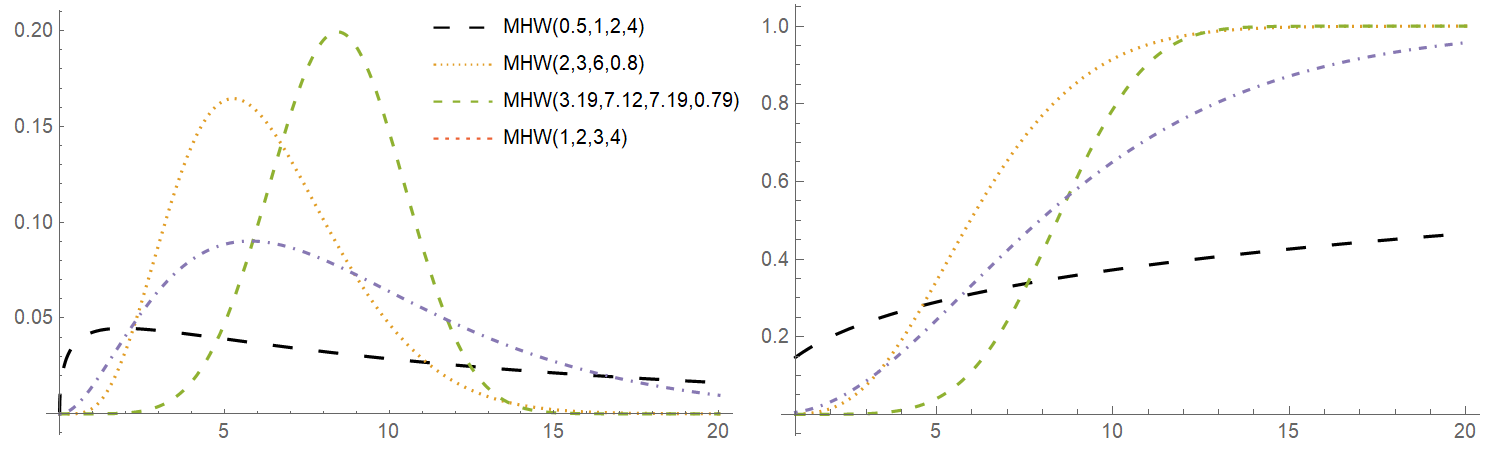}%
\caption{PDF and CDF of different distributions of $MHW(k,\lambda_{1}%
,\lambda_{2},...,\lambda_{n})$}%
\end{center}
\end{figure}

\subsection{The Mixed Hypoexponential Frechet Distribution}

The second distribution of the New Mixed Hypoexponential- G Family.

\begin{theorem}
Given $T\sim hypo\exp(\overrightarrow{\alpha})$, then $f_{T}(t)=\sum
\limits_{i=1}^{n}\frac{f_{X_{i}}(t)}{P_{i}},\ t>0$ where $X_{i}$ $\sim
Exp(\alpha_{i})$. Take $Y_{i}=X_{i}^{-c}\sim Frechet(\frac{1}{c},\alpha
_{i}^{c})$ where $\alpha_{i}>0,\ c>0.$ Then the Mixed Hypoexponential Frechet
distribution $Z,\ $i.e. $Z\sim MHF(k,\lambda_{1},\lambda_{2},...,\lambda
_{n}),\ k,\lambda_{i}>0,$ has PDF
\[
f_{Z}(t)=\sum\limits_{i=1}^{n}\frac{1}{PF_{i}}f_{Y_{i}}(t)
\]
where $Y_{i}$ $\sim Frechet(k,\lambda_{i})$ and $PF_{i}=\prod
\limits_{j=1,j\neq i}^{n}(1-\left(  \frac{\lambda_{i}}{\lambda_{j}}\right)
^{k})$.
\end{theorem}

\begin{corollary}
Let $Z\sim MHF(k,\lambda_{1},\lambda_{2},...,\lambda_{n})$ then
\[
f_{Z}(t)=\sum\limits_{i=1}^{n}\frac{1}{PF_{i}}f_{Y_{i}}(t)
\]
where $Y_{i}$ $\sim Frechet(k,\lambda_{i})$, $i=1,2,...,n$ be n random
variables with $\lambda_{i}\neq\lambda_{j}$ and $PF_{i}=\prod
\limits_{j=1,j\neq i}^{n}(1-\left(  \frac{\lambda_{i}}{\lambda_{j}}\right)
^{k}).$ Then according to Theorem \ref{cdf} $Z$ has the following CDF
\[
F_{Z}(t)=\sum\limits_{i=1}^{n}\frac{1}{PF_{i}}F_{Y_{i}}(t)=\frac{e^{-\left(
\frac{t}{\lambda_{i}}\right)  ^{-k}}}{\prod\limits_{j=1,j\neq i}^{n}(1-\left(
\frac{\lambda_{i}}{\lambda_{j}}\right)  ^{k})}%
\]

\end{corollary}

Next, we give some different graphs of PDF and CDF for some values of k and
$\lambda_{i}$, Showing the flexibility of adding extra parameters to a new distribution.%

\begin{figure}
[ptbh]
\begin{center}
\includegraphics[width=5in
]%
{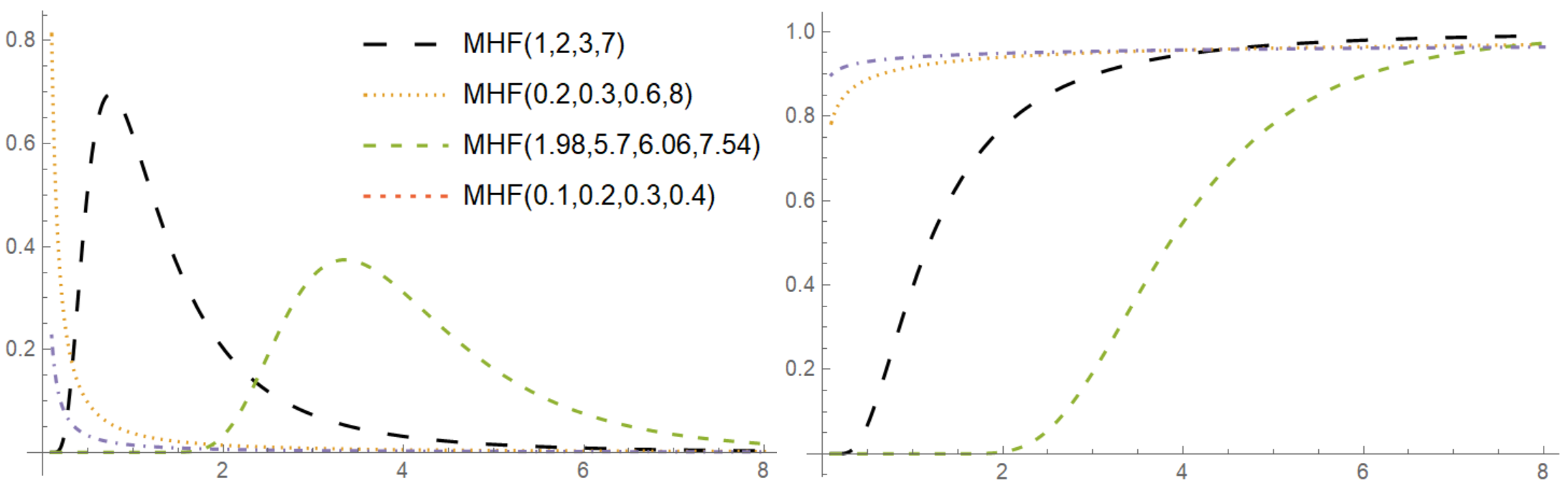}%
\caption{PDF and CDF of different distributions of MHF(k,$\lambda_{1}%
,\lambda_{2},\lambda_{3})$}%
\end{center}
\end{figure}

\subsection{The Mixed Hypoexponential Pareto Distribution}

\begin{theorem}
Let $T\sim hypo\exp(\overrightarrow{\alpha})$\ then $f_{T}(t)=\sum
\limits_{i=1}^{n}\frac{f_{X_{i}}(t)}{P_{i}},\ t>0$ where $X_{i}$ $\sim
Exp(\alpha_{i})$ and $P_{i}=\prod\limits_{j=1,j\neq i}^{n}(1-\frac{\alpha_{i}%
}{\alpha_{j}})$. Consider $Y_{i}=ce^{X_{i}}\sim Pareto(c,\alpha_{i})$ where
$\alpha_{i}>0,\ c>0.$ Then, our new distribution is the Mixed Hypoexponential
Pareto distribution $Z\sim MHT(k,\lambda_{1},\lambda_{2},...,\lambda_{n})$
with PDF
\[
f_{Z}(t)=\sum\limits_{i=1}^{n}\frac{1}{PT_{i}}f_{Y_{i}}(t)
\]
where $Y_{i}$ $\sim Pareto(k,\lambda_{i})$ and $PT_{i}=\prod\limits_{j=1,j\neq
i}^{n}(1-\frac{\lambda_{i}}{\lambda_{j}}).$
\end{theorem}

\begin{corollary}
Let $Z\sim MHT(k,\lambda_{1},\lambda_{2},...,\lambda_{n}).$ Then according to
Theorem \ref{cdf} the CDF of $Z$ is
\[
F_{Z}(t)=\sum\limits_{i=1}^{n}\frac{1}{PT_{i}}F_{Y_{i}}(t)=\sum\limits_{i=1}%
^{n}\frac{(1-(\frac{\lambda_{i}}{t})^{k})}{\prod\limits_{j=1,j\neq i}%
^{n}(1-\frac{\lambda_{i}}{\lambda_{j}})}%
\]

\end{corollary}

Here are the graphs of PDF and CDF for some values of k and $\lambda_{i}$,
Showing the flexibility of adding extra parameters to New Mixed
Hypoexponential Pareto distribution.%
\begin{figure}
[ptbh]
\begin{center}
\includegraphics[width=5in]%
{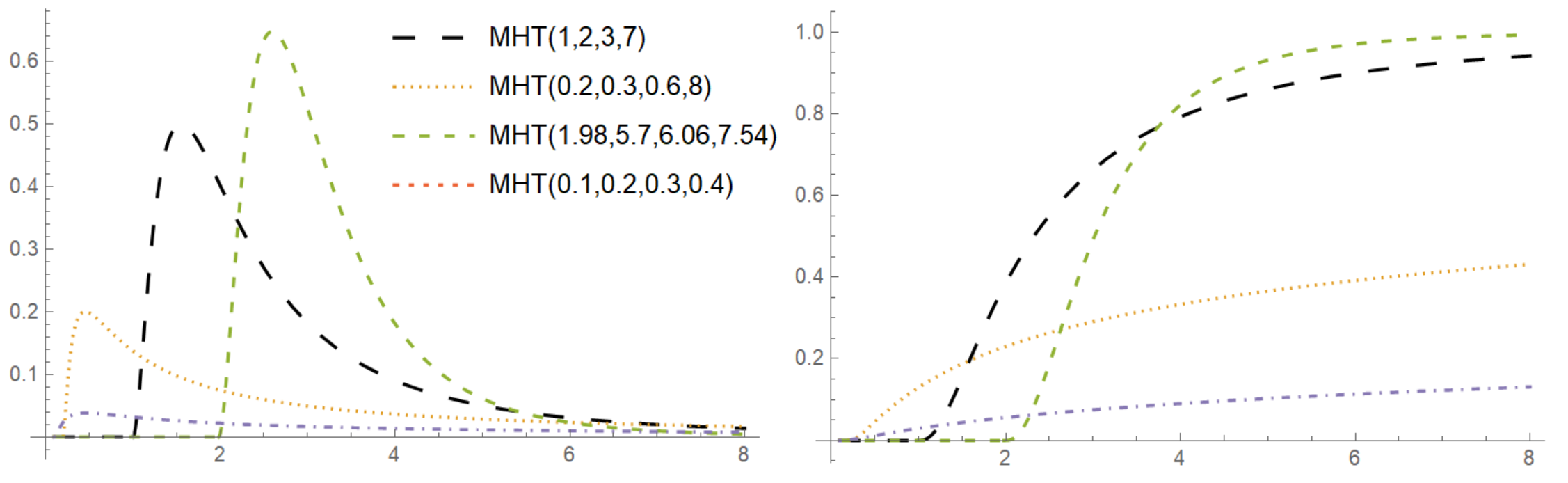}%
\caption{PDF\ and CDF of different distributions of MHT(k,$\lambda_{1}%
,\lambda_{2},\lambda_{3})$}%
\end{center}
\end{figure}

\subsection{The Mixed Hypoexponential Power Distribution}

Fourth is the Mixed Hypoexponential Power distribution.

\begin{theorem}
Let $T\sim hypo\exp(\overrightarrow{\alpha})$, then $f_{T}(t)=\sum
\limits_{i=1}^{n}\frac{f_{X_{i}}(t)}{P_{i}},\ t>0\ $where $X_{i}$ $\sim
Exp(\alpha_{i})$ Consider $Y_{i}=ce^{-X_{i}}\sim Power(\frac{1}{c},\alpha
_{i})\ $where $\alpha_{i}>0,\ c>0,$ domain $\left(  0,c\right)  .$ Then the
Mixed Hypoexponential Power distribution $Z\sim MHP(k,\lambda_{1},\lambda
_{2},...,\lambda_{n})$ has PDF
\[
f_{Z}(t)=\sum\limits_{i=1}^{n}\frac{1}{PP_{i}}f_{Y_{i}}(t)
\]
where $Y_{i}$ $\sim Power(k,\lambda_{i})$ and $PP_{i}=\prod\limits_{j=1,j\neq
i}^{n}(1-\frac{\lambda_{i}}{\lambda_{j}}).$
\end{theorem}

\begin{corollary}
Let $Z\sim MHP(k,\lambda_{1},\lambda_{2},...,\lambda_{n})$ then CDF of $Z$ is%
\[
F_{Z}(t)=\sum\limits_{i=1}^{n}\frac{1}{PP_{i}}F_{Y_{i}}(t)=\left\{
\begin{tabular}
[c]{ll}%
$\sum\limits_{i=1}^{n}\frac{(kt)^{\lambda_{i}}}{\prod\limits_{j=1,j\neq i}%
^{n}(1-\frac{\lambda_{i}}{\lambda_{j}})}$ & $0<t<\frac{1}{k}$\\
$\sum\limits_{i=1}^{n}\frac{1}{\prod\limits_{j=1,j\neq i}^{n}(1-\frac
{\lambda_{i}}{\lambda_{j}})}$ & $t>\frac{1}{k}$%
\end{tabular}
\right.
\]

\end{corollary}

The following are the graphs of PDF and CDF of Mixed Hypoexponential
distribution for some values of k and $\lambda_{i}$.%
\begin{figure}
[ptbh]
\begin{center}
\includegraphics[width=5in
]%
{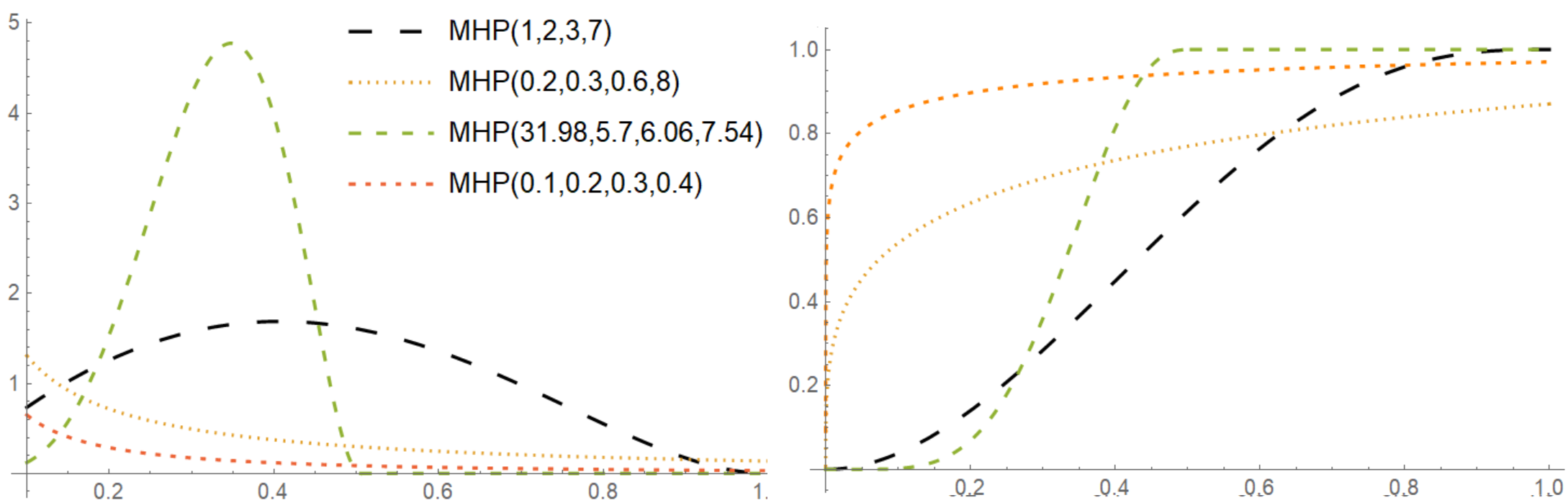}%
\caption{PDF and CDF of different distributions of MHP(k,$\lambda_{1}%
,\lambda_{2},\lambda_{3})$}%
\end{center}
\end{figure}

\subsection{The Mixed Hypoexponential Gumbel Distribution}

Next is the fifth distribution of the Mixed Hypoexponential - G Family

\begin{theorem}
Let $T\sim hypo\exp(\overrightarrow{\alpha})$, then $f_{T}(t)=\sum
\limits_{i=1}^{n}\frac{f_{X_{i}}(t)}{P_{i}},\ t>0$ Consider $Y_{i}=c\ln
X_{i}\sim Gumbel(-c\ln\left(  \alpha_{i}\right)  ,c)$ where $\alpha_{i}%
>0$,$\ c>0,$ domain $%
\mathbb{R}
.\ $Then, the Mixed Hypoexponential Gumbel\textbf{\ }distribution $Z\sim
MHG(k_{1},k_{2},...,k_{n},\lambda)$ has PDF
\[
f_{Z}(t)=\sum\limits_{i=1}^{n}\frac{1}{PG_{i}}f_{Y_{i}}(t)
\]
where $Y_{i}$ $\sim Gumbel(k_{i},\lambda)$ and $PG_{i}=\prod\limits_{j=1,j\neq
i}^{n}(1-e^{\frac{-k_{i}+k_{j}}{\lambda}}).$
\end{theorem}

\begin{corollary}
Let $Z\sim$ Mixed Hypoexponential Gumbel distribution MHG$(k_{1}%
,k_{2},...,k_{n},\lambda)$ then
\[
f_{Z}(t)=\sum\limits_{i=1}^{n}\frac{1}{PG_{i}}f_{Y_{i}}(t)
\]

where $Y_{i}$ $\sim Gumbel(k_{i},\lambda)$, $i=1,2,...,n$ be n random
variables with $k_{i}\neq k_{j}$ and $PG_{i}=\prod\limits_{j=1,j\neq i}%
^{n}(1-e^{\frac{-k_{i}+k_{j}}{\lambda}})$. Then,%
\[
F_{Z}(t)=\sum\limits_{i=1}^{n}\frac{1}{PG_{i}}F_{Y_{i}}(t)=\sum\limits_{i=1}%
^{n}\frac{1}{\prod\limits_{j=1,j\neq i}^{n}(1-e^{\frac{-k_{i}+k_{j}}{\lambda}%
})}e^{-e^{\frac{(t-k_{i})}{\lambda}}}.
\]

\end{corollary}

Next, we give the graphs of PDF and CDF for some values of $k_{i}$ and
$\lambda$ of the New Mixed Hypoexponential Gumbel distribution.%
\begin{figure}
[ptbh]
\begin{center}
\includegraphics[width=5in]%
{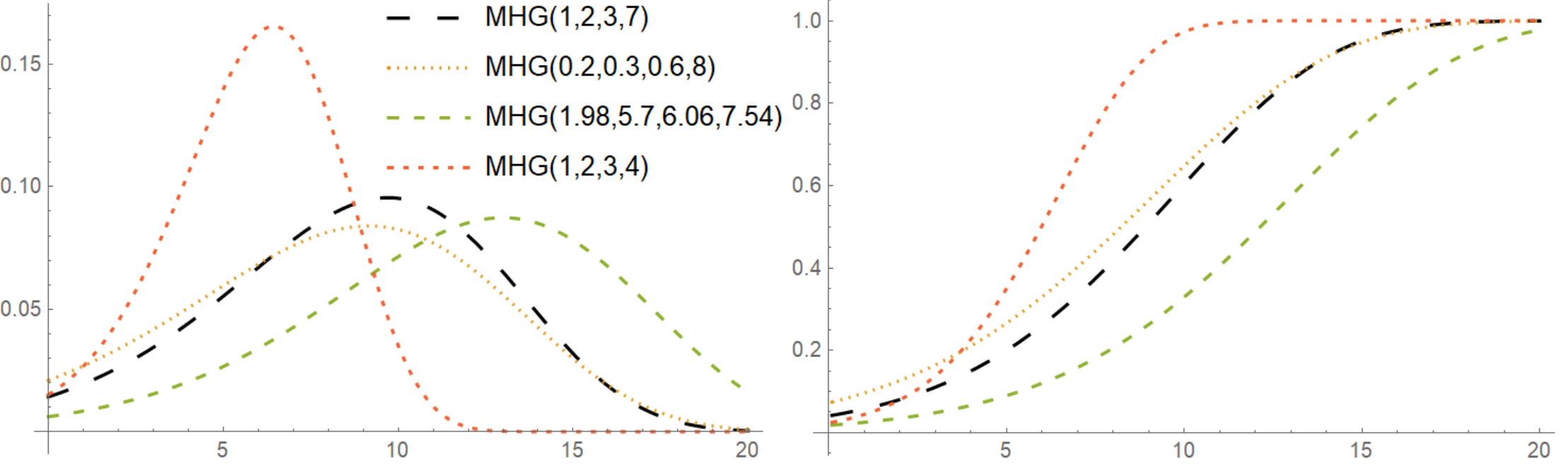}%
\caption{PDF and CDF of different distributions of MHG($k_{1},k_{2}%
,k_{3},\lambda)$}%
\end{center}
\end{figure}

\subsection{The Mixed Hypoexponential Extreme Value Distribution}

Finally, the last distribution of the Mixed Hypoexponential - G Family is The
Mixed Hypoexponential Extreme Value Distribution.

\begin{theorem}
Let $T\sim hypo\exp(\overrightarrow{\alpha})$, then $f_{T}(t)=\sum
\limits_{i=1}^{n}\frac{f_{X_{i}}(t)}{P_{i}},\ t>0$ where $X_{i}$ $\sim
Exp(\alpha_{i})$. Suppose $Y_{i}=c\ln X_{i}\sim ExtremeValue(-c\ln\left(
\alpha_{i}\right)  ,-c)$ where $\alpha_{i}>0,$ $c<0.$ Thus, the Mixed
Hypoexponential Extreme Value\textbf{\ }distribution $Z\sim MHE(k_{1}%
,k_{2},...,k_{n},\lambda)$ is the new distribution with PDF
\[
f_{Z}(t)=\sum\limits_{i=1}^{n}\frac{1}{PE_{i}}f_{Y_{i}}(t)
\]
where $Y_{i}$ $\sim ExtremeValue(k_{i},\lambda)$ and $PE_{i}=\prod
\limits_{j=1,j\neq i}^{n}(1-e^{\frac{k_{i}-k_{j}}{\lambda}}).$
\end{theorem}

Same as previous, and referring to Theorem \ref{cdf} as $Z\sim MHE(k_{1}%
,k_{2},...,k_{n}\lambda)$ then
\[
F_{Z}(t)=\sum\limits_{i=1}^{n}\frac{1}{PE_{i}}F_{Y_{i}}(t)=\sum\limits_{i=1}%
^{n}\frac{1}{\prod\limits_{j=1,j\neq i}^{n}(1-e^{\frac{k_{i}-k_{j}}{\lambda}%
})}e^{-e^{-\frac{(t-k_{i})}{\lambda}}}.
\]

and the graphs of the PDF and CDF for different values of $k_{i}$ and
$\lambda$ are as the following
\begin{figure}
[ptbh]
\begin{center}
\includegraphics[width=5in
]%
{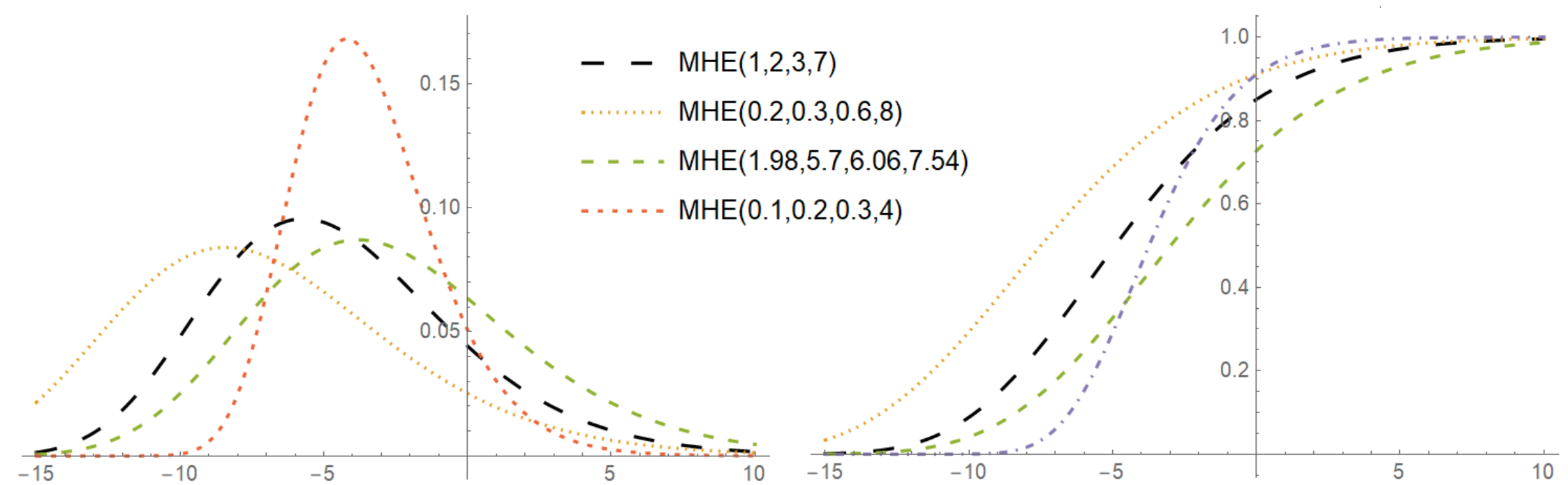}%
\caption{PDF and CDF of different distributions of MHE($k_{1},k_{2}%
,k_{3},\lambda)$}%
\end{center}
\end{figure}

Finally, we obtain the Mixed Hypoexponential-G family out of 6 different Mixed
T-G distributions. However the properties of each of these distribution that
are the CDF, MGF, hazard rate and reliability functions, and moment of order k
are obtained in the same manner as those of Mixed Hypoexponential Weibull Distribution.

\section{Conclusion}

New Mixed Hypoexponential G-Family is an example of New Mixed T-G Family in
which its distributions are derived by substituting the exponential
distribution in the Hypoexponential distribution by its inverse, scalar
multiple, k-th power, exponential, logarithm, and other transformations.
Distributions belonging to this family take a common general form of PDF which
is the New Mixed distribution form where their CDF, moment generating
function, reliability function hazard rate function, and their parameter
estimation can be determined easily according to the properties of New Mixed
distribution. substitutions in a distribution of the New Mixed form.

\end{document}